\documentclass[10pt,a4paper]{article}
\usepackage{fullpage}
\usepackage{amsfonts,amsmath,amssymb}
\usepackage{amsthm}
\usepackage{graphicx}\usepackage{sectsty}

\theoremstyle{plain}
\newtheorem{theorem}{Theorem}[section]
\newtheorem{lemma}[theorem]{Lemma}

\theoremstyle{definition}
\newtheorem{definition}[theorem]{Definition}

\theoremstyle{remark}

\numberwithin{equation}{section}

\sectionfont{\large}
\newenvironment{acknowledgement}[1][Acknowledgement
]{\begin{trivlist} \item[\hskip \labelsep {\bfseries
#1}]}{\end{trivlist}}

\begin{document}
\title{An Inverse Uniqueness of a Phaseless Scattering Problem \\by Zero-Crossings}
\author{Lung-Hui Chen$^1$}\maketitle\footnotetext[1]{Department of
Mathematics, National Chung Cheng University, 168 University Rd.
Min-Hsiung, Chia-Yi County 621, Taiwan. Email:
mr.lunghuichen@gmail.com. Fax:
886-5-2720497.}

\begin{abstract}
We discuss the inverse uniqueness problem in phaseless scattering by counting the zeros of its modulus of the scattering amplitude. The phase linearization of scattered wave field disturbs the originally uniform distribution of the zero set. There is a connection between the perturbation of the index of refraction and the zero distribution of the modulus.  We conclude the inverse uniqueness of the phaseless problem from  the point of view of interior transmission problem.
\end{abstract}
\section{Introduction}
In this paper, we study the inverse spectral uniqueness on the following scattering problem defined by the perturbation inside $\Omega:=\{x\in\mathbb{R}^{3}|\,|x|<R,\,R>0\}$.
\begin{eqnarray}\label{1.1}
\left\{%
\begin{array}{ll}
\Delta u(x)+k^2n(x)u(x)=0,\,x\in\mathbb{R}^3,\,\Im k\leq0;\vspace{5pt}\\\vspace{3pt}
u(x)=u^i(x)+u^s(x),\,x\in\mathbb{R}^3\setminus \Omega; \\
\lim_{|x|\rightarrow\infty}|x|\{\frac{\partial u^s(x)}{\partial |x|}-iku^s(x)\}=0.
\end{array}%
\right.
\end{eqnarray} 
Here, $u^{i}(x)=e^{-ikx\cdot\nu}$, and $u^s(x)=u^s(x,k,\nu)$ is the scattered wave field that depends on the frequency variable $k$ and the impinging direction $\nu\in\mathbb{S}^{2}$. In this paper, we follow the wave propagation theory established in M. V. Klibanov and V. G. Romanov \cite{Klibanov2}. In particular, we assume that 
\begin{eqnarray}
&&n(x)=1+\beta(x)\in C^{15}(\mathbb{R}^{3};\mathbb{R});\\
&&\beta(x)\geq0,\,\beta(x)=0,\,\mbox{ for }x\in\mathbb{R}^{3}\setminus \Omega.
\end{eqnarray}
It is shown in \cite{Klibanov2} that the problem~(\ref{1.1}) has a solution $u(x)$ in H\"{o}lder space $C^{16+\alpha}(\mathbb{R}^{3})$. The phaseless inverse scattering problem (PISP) is to find the index of refraction $n(x)$ if the  information is given or partially given in the following scattering data.
\begin{equation}\label{1.4}
 f(x,\nu,k):=|u^{s}(x,k,\nu)|^{2},\,(x,\nu,k)\in\mathbb{R}^{3}\times\mathbb{S}^{2}\times\mathbb{C}.
\end{equation}
The problem arises from the inverse scattering theory in many settings, e.g., electron microscopy, crystallography, medical imaging, and nano-optics. The historic review and the details can be found in \cite{Klibanov,Klibanov2,Klibanov3}.  We apply the zero distribution theory in complex variable theory \cite{Boas,Cartwright2,Koosis,Levin,Levin2} to study the uniqueness theorem on the index of refraction $n(x)$. The zero distribution of function $ f(x,\nu,k)$ is one kind of phaseless information on the perturbation. 
For every $\nu\in\mathbb{S}^{2}$, we denote $$S^{\pm}(\nu):=\{x\in\partial\Omega|\,x\cdot\nu\gtrless0 \}.$$
We state the following inverse uniqueness result.
\begin{theorem}\label{11}
Let $f^{j}(x,\nu,k)$ be the square modulus of the complex-valued scattered wave field generated by the index of refraction $n^{j}$, $j=1,2$. If $f^1(x,\nu,k)\equiv f^2(x,\nu,k)$ with the fixed pair $\pm\nu$ in $\mathbb{S}^{2}$ for all $x\in\mathbb{S}^{\pm}(\nu)$ in a neighborhood $U$ in $\Im k\leq0$, then $n^{1}\equiv n^{2}$.
\end{theorem}
The zero distribution theory in integral function theory \cite{Boas,Cartwright2,Koosis,Levin,Levin2} specifies the maximal zero-crossing density for the scattering data function $ f(x,\nu,k)$. Whenever the zero crossing density exceeds its theoretical maximal quantity, the function must be identically zero. This renders an inverse uniqueness on the scatterer, as we have discussed in \cite{Chen,Chen6,Chen7}. The assumption on the neighborhood $U$ in Theorem \ref{11} can be replaced by the assumptions on the numerical quantity of the zero-crossing density of $f(x,\nu,k)$.
\section{Lemmas}
There are two components in this paper: the Cartwright-Levinson type of theorems in entire function theory \cite{Boas,Cartwright2,Koosis,Levin,Levin2} and the phase-linearization constructed in \cite{Klibanov2}. In particular, we have the following asymptotic expansion under the assumptions mentioned in Introduction.
\begin{eqnarray}\label{2.1}
u(x,k,\nu)=A(x,\nu)\exp\{-ik\varphi(x,\nu)\}+\int_{\varphi(x,\nu)}^{\infty}\hat{v}(x,t,\nu)e^{-ikt}dt,\,\Im k=0.
\end{eqnarray}
Accordingly, the following high-energy expansion holds.
\begin{eqnarray}\label{2.2}
&&u^{s}(x,k,\nu)=A(x,\nu)\exp\{ik\varphi(x,\nu)\}-\exp\{ikx\cdot\nu\}+O(\frac{1}{k}),\,\Im k=0,\,\forall\nu\in
\mathbb{S}^{2},
\end{eqnarray}
which holds for $x$ in an arbitrary bounded domain, and
\begin{equation}\nonumber
A(x,\nu)=\exp\{-\frac{1}{2}\int_{{\Gamma(x,\nu)}}n^{-2}(\xi)\Delta_{\xi}\varphi(\xi,\nu)d\tau\}>0,
\end{equation}
where $\varphi(x,\nu)$ is the solution of the following problem
\begin{eqnarray}
&&|\nabla_{x}\varphi(x,\nu)^{2}|=n^{2}(x);\label{2.33}\\
&&\varphi(x,\nu)=x\cdot\nu,\,\forall x\cdot\nu\leq -B.\label{2.44}
\end{eqnarray}
Here we define
\begin{equation}
A(k):=A(x,\nu)\exp\{ik\varphi(x,\nu)\}-\exp\{ikx\cdot\nu\},\label{22.2}
\end{equation}
which would contribute to the zero-crossings of $f(x,\nu,k)$.
For far-field behavior, we have
\begin{equation}\label{233}
u^s(x,k,\nu)=\frac{e^{ik|x|}}{|x|}u_\infty(\hat{x};\nu,k)+O(\frac{1}{|x|^{\frac{3}{2}}}),
\end{equation}
which holds uniformly for all $\hat{x}:=\frac{x}{|x|}$, $x\in\mathbb{R}^3$, and $u_\infty(\hat{x};\nu,k)$ is known as the scattering amplitude in the literature \cite{Colton2,Isakov,Lax,Melrose}. In this paper, we adopt the convention that $u_\infty(\hat{x};\nu,k)$ is defined analytically in $\Im k\leq 0$, and extends meromorphically from $\Im k\leq0$ to $\mathbb{C}$.
\begin{lemma}\label{S}
The scattered wave field $u^{s}(x,k,\nu)$ in~(\ref{233}) is defined meromorphically in $\mathbb{C}$ with poles in $\Im k >0$ except a finite number of purely imaginary $k$'s  that  $k^{2}$ are the negative eigenvalues of~(\ref{1.1}).  In particular, the poles of $u_\infty(\hat{x};\nu,k)$ are located as the mirror images of its zeros  to the real axis. 
\end{lemma}
\begin{proof}
This is well-known in scattering theory.  Let us refer to \cite{Lax,Melrose}.
\end{proof}
The error term in~(\ref{2.2}) comes from the decaying rate of $\int_{\varphi(x,\nu)}^{\infty}\hat{\nu}(x,t,\nu)e^{-ikt}dt$. Taking advantage of the decaying rate, we prove the following lemma.
\begin{lemma}\label{22}
The expansion~(\ref{2.2}) holds in $-\infty<\Im k<C$ for some $C>0$ in 
$\mathbb{C}$.
\end{lemma}
\begin{proof}
We recall Theorem 1 provided in \cite{Klibanov2}, which says the solution $v$ of the system
\begin{eqnarray*}
\left\{%
\begin{array}{ll}
n^{2}(x)v_{tt}-\Delta v=0,\,(x,t,\nu)\in\mathbb{R}^{3}\times\mathbb{R}\times\mathbb{S}^{2};\vspace{4pt}\\\vspace{4pt}
v(x,t,\nu)=\delta(t-x\cdot \nu)+\overline{v}(x,t,\nu); \\
\overline{v}(x,t,\nu)\equiv0,\,t<-B,
\end{array}%
\right.
\end{eqnarray*} 
can be represented in the form
\begin{equation}\label{2.3}
v(x,t,\nu)=A(x,y)\delta(t-\varphi(x,\nu))+\hat{v}(x,t,\nu)H(t-\varphi(x,\nu)),\,(x,t)\in\mathbb{R}^{3}\times(-\infty,T),
\end{equation}
where the formula holds for $x$ in arbitrary bounded $T$, $H(t-\varphi(x,\nu))$ is the Heaviside function, and $A(x,y)$ is as defined  in \cite[(4.6)]{Klibanov2}. The propagating wave $v(x,t,\nu)$ decays exponentially for large $t$ for bounded $x$. That is,
\begin{equation}\label{2.4}
|v(x,t,\nu)|\leq Ce^{-Ct},\,\mbox{ as }t\rightarrow\infty,
\end{equation}
for some constant $C>0$.
Here, we refer~(\ref{2.4}) to \cite[(4.23)]{Klibanov2}. According to~(\ref{2.3}) and~(\ref{2.4}), the same decaying rate~(\ref{2.4}) holds for the twice differentiable exponentially-decaying $\hat{v}(x,t,\nu)$. Hence,
\begin{equation}\label{2.6}
\int_{\varphi(x,\nu)}^{\infty}\hat{v}(x,t,\nu)e^{-ikt}dt=\int_{\varphi(x,\nu)}^{\infty}e^{-i\Re k t}[\hat{v}(x,t,\nu)e^{\Im k t}]dt,
\end{equation}
in which the function $\hat{v}(x,t,\nu)e^{\Im k t}$ is integrable over $t\in[\varphi(x,\nu),\infty]$ if $\Im k<C$. Thus, the Fourier transform $$I(\Re k):=\int_{\varphi(x,\nu)}^{\infty}e^{-i\Re k t}[\hat{v}(x,t,\nu)e^{\Im k t}]dt$$ makes sense for some $\Im k\leq C$, $C>0$, and $\Re k\in(-\infty,\infty)$. Applying Riemann-Lebesgue Lemma, we observe that the integral $I(\Re k)$ vanishes for large $\Re k$. Integrating by parts to~(\ref{2.6}), we deduce that
\begin{eqnarray*}
&&u(x,k,\nu)=A(x,y)e^{-ik\varphi(x,\nu)}+\frac{e^{-ik\varphi(x,\nu)}}{ik}\hat{v}_{+}(x,\varphi(x,\nu),\nu))+\frac{1}{ik}\int_{\varphi(x,\nu)}^{\infty}e^{-i\Re k t}[\hat{v}_{t}(x,t,\nu)e^{yt}]dt,
\end{eqnarray*}
in which the behavior of $\hat{v}_{t}(x,t,\nu)$ and its derivatives are found in \cite[p.\,362]{Vainberg} and the Fourier transform decays like $I(\Re k)$ as discussed above.

\end{proof}
The zeros of $u^{s}(x,k,\nu)$ is a class of phaseless information. This is called zero-crossing method in \cite{Hunt}, and therein we have a short introduction to Cartwright theory as well. Let us study the qualitative theory of the zero set in this paper.
\begin{definition}\label{23}
Let $F(z)$ be an integral function of order $\rho$, and let
$N(F,\alpha,\beta,r)$ denote the number of the zeros of $F(z)$
inside the angle $[\alpha,\beta]$ and $|z|\leq r$. We define the zero 
density function as
\begin{equation}\nonumber
\Delta_F(\alpha,\beta):=\lim_{r\rightarrow\infty}\frac{N(F,\alpha,\beta,r)}{r^{\rho}},
\end{equation}
and
\begin{equation}\nonumber
\Delta_F(\beta):=\Delta_F(\alpha_0,\beta),
\end{equation}
with some fixed $\alpha_0\notin E$ such that $E$ is at most a
countable set \cite{Boas,Cartwright2,Koosis,Levin,Levin2}. We can define a similar notation for a set of zeros.
\end{definition}
\begin{lemma}
Let $S(\alpha,s,h)$ be the strip containing the positive real axis starting $\Re k\geq\alpha>0$, $|\Im k|\leq h$, and $\alpha\leq\Re k\leq\alpha+s$.
The zeros of the analytic function $A(k)$ in~(\ref{22.2}) are located in a suitable $S(\alpha,s,h)$ with density $\frac{\varphi(x,\nu)-x\cdot\nu}{2\pi}$.
\end{lemma}
\begin{proof}
To estimate the zero asymptotics of $A(k)$, we consider the following zero counting theorem 
\cite{Dickson,Dickson2,Levin}. 
\begin{theorem}[Dickson \cite{Dickson,Dickson2}]\label{25} Let
\begin{equation}\nonumber
R(\alpha,s,h):=\{z=x+iy\in\mathbb{C}|\,|x|\leq
h,\,y\in[\alpha,\alpha+s]\};
\end{equation}
\begin{equation}\nonumber
N_g(R(\alpha,s,h)):=\{\mbox{the number of zeros of }g(z)\mbox{ in
}R(\alpha,s,h)\},
\end{equation}
in which $$g(z)=\sum_{j=1}^nA_j e^{\omega_jz},$$ where $z=x+iy$,
$A_j\neq0$, $\omega_1<\omega_2<\cdots<\omega_n$. Then, there
exists $K>0$ such that
\begin{enumerate}
    \item each zero of $g$ is in $|x|<K$;
\item for each pair of reals $(\alpha,s)$ with $s>0$,
\begin{equation}
|N_g(R(\alpha,s,K))-s(\omega_n-\omega_1)/(2\pi)|\leq n-1.
\end{equation}
\end{enumerate}
\end{theorem}

\par
In our case, the application is straightforward. Let us set $n=2$ with the invariant $\omega_{2}:=\varphi(x,\nu)$ and $\omega_{1}:=x\cdot\nu>0$. We have only one interval to consider, and note here that $\varphi(x,\nu)-x\cdot\nu>0$ if the index of refraction $1+\beta(x)\not\equiv1$. We refer the details to \cite[Sec. \,3]{Klibanov2}. For our application, we consider $z\mapsto iz$, and then deduce that
\begin{equation}\nonumber
|N_{A(k)}S(\alpha,s,h))-s|\varphi(x,\nu)-x\cdot\nu|/(2\pi)|\leq1.
\end{equation}

\end{proof}

\begin{lemma}\label{266}
Let $T(\alpha,s,h)$ be the strip containing the positive real axis starting $\Re k\geq\alpha>0$, $-h\leq\Im k\leq C$, and $\alpha\leq\Re k\leq\alpha+s$ with suitable $h>0$ and the constant $C$ from Lemma \ref{22}. We denote the zeros of $u^{s}(x,\nu,k)$ in $T(\alpha,s,h)$ by $N(\alpha,s,h)$. Then,
$$|N_{u^{s}(x,k,\nu)}(T(\alpha,s,h))-s(\varphi(x,\nu)-x\cdot\nu)/(2\pi)|\leq1.$$
\end{lemma}
\begin{proof}
From the Lemma \ref{22} and~(\ref{2.2}), we apply the following inequality
\begin{equation}\label{299}
|u^{s}(x,k,\nu)-A(k)|\leq O(\frac{1}{k}),\,\Im k\leq C,
\end{equation}  
where the big O-term is defined in~(\ref{2.1}). Moreover,~(\ref{299}) implies the zeros of $A(k)$ are near the ones of $u^{s}(x,k,\nu)$ for large $k$ by the continuity. We have shown that zeros of $A(k)$ is distributed in a suitable $S(\alpha,h,s)$
with suitable width $h>0$ and $\alpha>0$, so on the boundary of $T(\alpha,h,s)$, where $A(k)$ is away from zero, we deduce from~(\ref{299}) and Lemma \ref{S} that
\begin{equation}
|u^{s}(x,k,\nu)-A(k)|<|A(k)|.  
\end{equation}  
Thus the lemma is proved by Rouch\'{e}'s theorem in complex analysis. 

\end{proof}
\begin{lemma}\label{26}
The following asymptotic identities hold.
\begin{eqnarray*}
&&\Delta_{A(k)}(-\pi,0)=\Delta_{u^{s}(x,k,\nu)}(-\pi,0)=\frac{[\varphi(x,\nu)-x\cdot\nu]}{\pi};\\
&&\Delta_{f(x,\nu,k)}(-\pi,0)=\frac{2[\varphi(x,\nu)-x\cdot\nu]}{\pi}.
\end{eqnarray*}
\end{lemma}
\begin{proof}
The first identity is deduced from Definition \ref{23} and Lemma \ref{266}. Because $$ f(x,\nu,k)=u^{s}(x,\nu,k)\overline{u^{s}(x,\nu,k)},$$ and $u^{s}(x,\nu,k)$ and $\overline{u^{s}(x,\nu,k)}$ share the same zero set, $f(x,\nu,k)$ has twice the density as $u^{s}(x,\nu,k)$ or $A(k)$. The lemma is thus proven.

\end{proof}
Here we interpret the linearized term $\varphi(x,\nu)-x\cdot\nu$ as a Weyl's type of spectral invariant to the problem~(\ref{1.1}) considering the following identity.
\begin{equation}
\varphi(x,\nu)-x\cdot\nu=\int_{L(x,\nu)}\beta(\xi)d\xi,\,\forall\nu\in\mathbb{S}^{2},\,\forall x\in\mathbb{S}^{+}(\nu),
\end{equation}
where $L(x,\nu)$ is the line segment connecting $x$ and $x-2(x\cdot\nu)\nu$. Surely, we refer the construction to \cite[(6.5)]{Klibanov2}. Moreover, let us examine the problem~(\ref{2.33}) and~(\ref{2.44}). Whenever the observation position $x$ and the incident angle $\nu$ are given, the asymptotic properties in Lemma \ref{26} construct a connection between the index of refraction $n$ and  $\varphi(x,\nu)$ as shown in~(\ref{2.33}) and~(\ref{2.44}).

\section{A Proof of Theorem \ref{1.1}}

Starting with the assumption in Theorem \ref{11} that, for $U\subset\{k\in\mathbb{C}|\,\Im k\leq0\}$, we have
$$f^1(x,\pm\nu,k)=f^2(x,\pm\nu,k),\,x\in S^{\pm}(\nu).$$
Using the analytic continuation property of real-valued functions $f^{j}=\{\Re [u^{j}]^{s}\}^{2}+\{\Im [u^{j}]^{s}\}^{2}$, in which $[u^{j}]^{s}$ is denoted as the scattered wave field defined by the index of refraction $n^{j}$, $j=1,2$, then we can deduce that
$$f^1(x,\pm\nu,k)\equiv f^2(x,\pm\nu,k),\,x\in S^{\pm}(\nu),\,\Im k\leq0.$$
The boundary $\partial \Omega$ is starlike, so 
$\overline{S^{+}(\nu)\cup S^{-}(\nu)}=\partial\Omega$ for the given $\nu$.
Taking a square root and considering the continuity of the solutions over $x$, so we deduce
\begin{equation}\label{3.1}
|[u^{1}]^{s}(x,k,\nu)|\equiv |[u^{2}]^{s}(x,k,\nu)|,\,x\in\partial\Omega,\,\Im k\leq0.
\end{equation}
Therefore, we deduce from~(\ref{3.1}) that
$[u^{1}]^{s}(x,k,\nu)$ and $[u^{2}]^{s}(x,k,\nu)$ share the same zero set, say, $\mathcal{Z}(x,\nu)$ in $\Im k\leq0$ for the fixed $\nu\in\mathbb{S}^{2}$. From Lemma \ref{26}, the zero density is specified as
\begin{equation}\label{3.2}
\Delta_{\mathcal{Z}(x,\nu)}=\frac{\varphi(x,\nu)-x\cdot\nu}{\pi}.
\end{equation}
According to~(\ref{3.1}) and the fact that $[u^{j}]^{s}$ and $|[u^{j}]^{s}|$ share the same zero set, so $[u^{1}]^{s}/[u^{2}]^{s}$ is analytic with {\bf no zero }in $\Im k\leq0$. Moreover, $$\ln \{[u^{1}]^{s}(x,k,\nu)/[u^{2}]^{s}(x,k,\nu)\}=\ln |[u^{1}]^{s}(x,k,\nu)/[u^{2}]^{s}(x,k,\nu)|+i\arg \{[u^{1}]^{s}(x,k,\nu)/[u^{2}]^{s}(x,k,\nu)\}.$$ From~(\ref{3.1}), we obtain that
$$\ln \{[u^{1}]^{s}(x,k,\nu)/[u^{2}]^{s}(x,k,\nu)\}=i\arg \{[u^{1}]^{s}(x,k,\nu)/[u^{2}]^{s}(x,k,\nu)\},$$
which is purely imaginary. It is known from complex analysis that $\ln \{[u^{1}]^{s}(x,k,\nu)/[u^{2}]^{s}(x,k,\nu)\}$ is a constant. That is, $[u^{1}]^{s}(x,k,\nu)/[u^{2}]^{s}(x,k,\nu)=e^{i\gamma}$ for some real constant $\gamma$.
Considering the asymptotics~(\ref{2.2}),
\begin{eqnarray*}
\frac{[u^{1}]^{s}(x,k,\nu)}{[u^{2}]^{s}(x,k,\nu)}=
\frac{A^{1}(x,\nu)\exp\{ik\varphi^{1}(x,\nu)\}-\exp\{ikx\cdot\nu\}+O(\frac{1}{k})}{A^{2}(x,\nu)\exp\{ik\varphi^{2}(x,\nu)\}-\exp\{ikx\cdot\nu\}+O(\frac{1}{k})},\,\Im k=0.
\end{eqnarray*}
From~(\ref{3.1}) and~(\ref{3.2}), we obtain $\varphi^{1}(x,\nu)=\varphi^{2}(x,\nu)$ for fixed $(x,\nu)$. Given that $A^{1}(x,y)$ and $A^{2}(x,y)$ are real, we deduce that $A^{1}(x,y)=A^{2}(x,y)$, $\gamma=0$, and
\begin{eqnarray}
&&[u^{1}]^s(x,k,\nu)=[u^{2}]^s(x,k,\nu),\,x\in \partial \Omega,\, \Im k\leq0;\\
&&u^{1}(x,k,\nu)=u^{2}(x,k,\nu),\,x\in \partial \Omega,\, \Im k\leq0,\label{3.4}
\end{eqnarray}
where $u^{j}$, $j=1,2$, satisfies the Helmholtz's equation outside $\Omega$, and then enjoys the property of analytic continuation.
\par
Let us set the observation data
\begin{eqnarray*}
&&w(x;k):=u^1(x,k,\nu);\\
&&v(x;k):=u^2(x,k,\nu).
\end{eqnarray*}
Most importantly, for each zero-crossing data from~(\ref{1.1}) and~(\ref{3.4})
we deduce the following interior transmission problem \cite{Chen,Chen6,Chen7,Colton,Colton2,Mc,S} that holds for $\Im k\leq0$, and
\begin{eqnarray*}
\left\{%
\begin{array}{ll}
    \Delta w(x;k)+k^2n^1w(x;k)=0,  & x\in\mathbb{R}^3 ; \vspace{3pt}\\\vspace{3pt}
    \Delta v(x;k)+k^2n^2v(x;k)=0, & x\in\mathbb{R}^3; \\\vspace{3pt}
    w(x;k)=v(x;k), & x\in\partial\Omega; \\\vspace{3pt}
    \frac{\partial w(x;k)}{\partial n}=\frac{\partial v(x;k)}{\partial n},&  x\in\partial\Omega,
\end{array}%
\right.
\end{eqnarray*}
where $n$ is the unit outer normal. By applying the analytic continuation property of the Helmholtz's equation \cite{Colton2,Isakov}  outside $\Omega$, we deduce that, for $\Im k\leq0$, the following system holds.
\begin{eqnarray}\label{3.5}
\left\{%
\begin{array}{ll}
    \Delta w(x;k)+k^2n^1w(x;k)=0,  & x\in\mathbb{R}^3 ; \vspace{3pt}\\\vspace{3pt}
    \Delta v(x;k)+k^2n^2v(x;k)=0, & x\in\mathbb{R}^3; \\\vspace{3pt}
    w(x;k)=v(x;k), & x\in\mathbb{R}^3\setminus \Omega; \\\vspace{3pt}
    \frac{\partial w(x;k)}{\partial n}=\frac{\partial v(x;k)}{\partial n},&  x\in\mathbb{R}^3\setminus \Omega.
\end{array}%
\right.
\end{eqnarray}
Firstly, it is well-known that the spectrum of~(\ref{3.5}) is a discrete set in $\mathbb{C}$ \cite{Chen,Chen6,Colton,Colton2,Mc,S} if $n^{1}\not\equiv n^{2}$. Secondly, we have provided a spectral analysis for each incident angle $\nu$ in the previous section. Any quantitative assumption that exceeds the maximal zero crossing density leads to a proof on the uniqueness on the index of refraction. Thirdly, we may apply the inverse uniqueness of~(\ref{3.5}) \cite{Chen,Chen6,Chen7} in which we merely assume that $n^{j}\in C^{2}(\mathbb{R}^{3})$, $j=1,2$, so in any case we conclude that $n^{1}\equiv n^{2}$.
\begin{acknowledgement}
The author would like to thank Prof. M.V. Klibanov for providing the manuscript \cite{Klibanov2}.
\end{acknowledgement}


\begin{thebibliography}{widest-label}

\bibitem{Boas}R. P. Boas, Entire Functions, Academic Press, New York, 1954.
\bibitem{Cartwright2}M. L. Cartwright, Integral Functions,
Cambridge University Press, Cambridge, 1956.

\bibitem{Chen}L. -H. Chen, An uniqueness result with some density theorems with interior transmission eigenvalues, Applicable Analysis, V. 94,  8, 1527--1544 (2015).
\bibitem{Chen6}L. -H. Chen, An inverse uniqueness on the index of refraction with a transition region, Applicable Analysis, V. 95, 3, 545--561 (2016).
\bibitem{Chen7}L. -H. Chen, , An inverse uniqueness in interior transmission problem and its eigenvalue tunneling in simple domain, Advances in Mathematical Physics, Volume 2016, Article ID 2438253  (2016), Hidawi. 
\bibitem{Colton} D. Colton and P. Monk, The inverse scattering problem for time-harmonic acoustic waves in an inhomogeneous medium, Q. Jl. Mech. appl. Math. Vol. 41, 97--125 (1988).

\bibitem{Colton2}D. Colton and R. Kress, Inverse Acoustic and Electromagnetic Scattering Theory,
2rd ed. Applied Mathematical Science, V. 93, Springer--Verlag, Berlin,  2013.
\bibitem{Dickson}D. G. Dickson, Zeros of exponential sums, Proc. Amer. Math. Soc, 16, 84--89 (1965).
\bibitem{Dickson2}D. G. Dickson, Expansions in Series of Solutions
of Linear Difference-Differential and Infinite Order Differential
Equations with Constant Coefficients, Memoirs of the American
Mathematical Society, Rhode Island, USA, No. 23, 1957.

\bibitem{Hunt}N. Hurt, Phase Retrieval and Zero Crossings: Mathematical Methods in Image Construction, Kluwer Academic Publishers, Boston, 1989.
\bibitem{Isakov}V. Isakov, Inverse Problems for Partial Differential Equations, Applied Mathematical Sciences, V. 127, Springer-Verlag, New York, 1998.
\bibitem{Klibanov}M. V. Klibanov, P. E. Sacks, and A. V. Tikhonravov, The phase retrieval problem, Inverse Problems, 11, 1--28 (1995).  
\bibitem{Klibanov2}M. V. Klibanov and V. G. Romanov, Two reconstruction procedures for a 3D phaseless inverse scattering problem for the generalized Helmholtz equation, Forthcoming in ''Inverse Problems''.
\bibitem{Klibanov3}M. V. Klibanov, Phaseless inverse scattering problem in 3-d,  arXiv:1303.0923v1.
\bibitem{Koosis}P. Koosis, The Logarithmic Integral I, Cambridge University Press, New York, 1997.
\bibitem{Lax}P. D. Lax and R. S. Phillips, Scattering Theory, Academic Press, Boston, 1989.
\bibitem{Levin}B. Ja. Levin, Distribution of Zeros of Entire
Functions, revised edition, Translations of Mathematical
Monographs, American Mathematical Society, 1972.
\bibitem{Levin2}B. Ja. Levin, Lectures on Entire Functions,
Translation of Mathematical Monographs, V. 150, AMS, Providence, 1996.
\bibitem{Mc}J. R. McLaughlin and P. L. Polyakov, On the uniqueness of a spherically symmetric speed of sound from transmission eigenvalues, Jour. Differential Equations, 107, 351--382 (1994).
\bibitem{Melrose}R. B. Melrose, Geometric Scattering Theory, Cambridge University Press, New York, 1995.
\bibitem{S}J. Sylvester, Discreteness of transmission eigenvalues via upper triangular compact operators, SIAM Journal on Mathematical Analysis, Vol. 44, No.1, 341--354 (2012).
\bibitem{Vainberg}B. R. Vainberg, Asymptotic Methods in Equations of Mathematical Physics, Gordon and Breach Science Publishers, New York, 1989.
\end{thebibliography}
\end{document}